\documentclass[11pt,letterpaper,reqno]{amsart}
\usepackage{graphicx}
\usepackage{amsmath}
\usepackage{amssymb}
\usepackage{amsfonts}
\usepackage{amsrefs}
\usepackage{enumerate}
\usepackage{cancel}
\usepackage[mathscr]{eucal}
\usepackage{mathtools}

\usepackage[usenames]{color}

\newcommand{\R}{\mathbb{R}}

\newcommand{\ol}{\overline}

\newcommand{\M}{\mathcal{M}}
\newcommand{\F}{\mathcal{F}}
\renewcommand{\L}{\mathcal{L}}
\newcommand{\Z}{\mathcal{Z}}

\newtheorem{thm}{Theorem}
\newtheorem{lemma}[thm]{Lemma}
\newtheorem{prop}[thm]{Proposition}

\newtheorem{conj}[thm]{Conjecture}
\newtheorem{question}[thm]{Question}

\theoremstyle{definition}
\newtheorem{definition}[thm]{Definition}

\theoremstyle{remark}
\newtheorem{remark}{Remark}

\newtheorem*{outline}{Outline}

\DeclareMathOperator{\Hess}{Hess}
\DeclareMathOperator{\tr}{tr}
\DeclareMathOperator{\Ric}{Ric}
\DeclareMathOperator{\capac}{cap}

\newcommand{\C}{\mathcal{C}}
\renewcommand{\M}{\mathcal{M}}

\DeclareMathOperator{\Div}{div}

\begin{document}

\title{Maximum capacity of Bartnik data and a generalization of static metrics}

\author{Jeffrey L. Jauregui}
\address{Dept. of Mathematics,
Union College, 807 Union St.,
Schenectady, NY 12308,
United States}
\email{jaureguj@union.edu}
\date{\today}

\maketitle 

\begin{abstract}
Inspired by R. Bartnik's mass minimization problem in general relativity, we investigate a dual problem of maximizing the capacity among asymptotically flat extensions (with nonnegative scalar curvature) of some fixed two-dimensional boundary data. Using the method of Lagrange multipliers on the constraint space of scalar-flat extensions, we derive the variational condition satisfied by a maximal capacity extension. The resulting equation 
is an inhomogeneous generalization of the well-known static equation, now coupled with the Baird--Eells stress-energy tensor for a harmonic function. We analyze these ``harmonic-static'' metrics in a local sense, proving they have constant scalar curvature and serve as  critical points for a metric-dependent Dirichlet energy functional. We conclude with a number of open questions.
\end{abstract}

\section{Introduction}
\label{sec_intro}

In \cite{Ba1}, R. Bartnik proposed the problem of finding the optimal external configuration (from the point of view of general relativity) of a given compact region (i.e., a compact Riemannian 3-manifold with boundary) $\Omega$. By ``external configuration,'' we mean an asymptotically flat 3-manifold $M$ that glues $\partial M$ to $\partial \Omega$  (subject to natural constraints, explained later); ``optimal'' is taken here to mean the least possible total (ADM) mass. This smallest value of the total mass defines a so-called quasi-local mass of $\Omega$,  called the Bartnik mass. This definition, along with Bartnik's conjecture that optimal configurations exist, have been quite influential in the last few decades in the study of  general relativity and geometric analysis.

In this paper we propose an analogous optimal configuration problem based instead on maximizing the capacity of the extension $M$ rather than minimizing its total mass. Since the capacity  (recalled below)  is defined as an infimum, this is effectively a maximin problem for finding an ``efficient'' extension of $\Omega$. 

Before proceeding further, we recall a number of details to make the discussion more precise. We will generally consider only the role played by the boundary of the region $\Omega$, beginning with the definition:
\begin{definition}
\emph{Bartnik data} is a triple $(\Sigma, \gamma, H)$, where $\Sigma$ is a smooth, connected, orientable closed surface, $\gamma$ is a smooth Riemannian metric on $\Sigma$, and $H\geq 0$ is a smooth function on $\Sigma$.
\end{definition}

Next, we recall the definition of asymptotically flat and ADM mass:
\begin{definition}
\label{def_AF}
An \emph{asymptotically flat} (AF) 3-manifold is a smooth, connected Riemannian 3-manifold $(M,g)$ (with or without boundary) with integrable scalar curvature $R_g$, for which there exists a compact set $K \subset M$, a closed ball $B(0,r) \subset \R^3$, and a diffeomorphism $\Phi: M \setminus K \to \R^3 \setminus B(0,r)$ such that in the coordinate chart induced by this diffeomorphism,
$$|g_{ij} - \delta_{ij}| \leq C |x|^{-\tau}, \qquad |\partial_k g_{ij}| \leq C |x|^{-\tau-1}, \qquad |\partial_k \partial_l g_{ij}| \leq C |x|^{-\tau-2},$$
for constants $C > 0$ and $\tau \in (\frac{1}{2}, 1]$ (where $\tau$ is called the \emph{order}).
\end{definition}

\begin{definition}
The \emph{ADM mass} \cite{adm}  of an asymptotically flat 3-manifold $(M,g)$ (proven to be well-defined by Bartnik  \cite{Ba0} and Chru\'sciel \cite{Ch}) is:
$$m_{ADM}(M,g) = \frac{1}{16\pi} \lim_{r \to \infty} \int_{S_r} \sum_{i,j=1}^3\left( \partial_ig_{ij} - \partial_j g_{ii}\right)\frac{x^j}{r} dA ,$$
where $dA$ is the induced area form on the coordinate sphere $S_r=\{|x|=r\}$ with respect to $\delta_{ij}$, using coordinates as in Definition \ref{def_AF}.
\end{definition}

We now can define asymptotically flat extensions.
\begin{definition}
\label{def_extension}
An asymptotically flat (AF) \emph{extension} of Bartnik data $(\Sigma, \gamma, H)$ is an AF 3-manifold $(M,g)$ with boundary such that 1) $M$ is diffeomorphic to the closed region exterior to a smooth, closed, embedded surface in $\R^3$ and
2) the boundary $\partial M$ isometric to $(\Sigma,\gamma)$. The AF extension is \emph{admissible} if 3) $R_g$ is nonnegative and 4)  the boundary condition $H \geq H_{\partial M}$ holds, where we use the given isometry from $\Sigma$ to $\partial M$ to compare $H$ to the mean curvature $H_{\partial M}$ of $\partial M$ in the direction pointing into $M$.
\end{definition}

Variants of the definition of AF extension and admissibility appear in the literature (for example, more general topology on $M$ is often allowed and sometimes the stricter boundary condition $H=H_{\partial M}$ is imposed). Within the context of the Bartnik mass, closed minimal surfaces (``horizons'') are often disallowed in the definition of admissible, a point we return to shortly.

\begin{remark}
Suppose $(\Omega,g_\Omega)$ is a compact orientable Riemannian 3-manifold with connected boundary. Let $\gamma$ be the induced metric on $\partial \Omega$ and $H$ the mean curvature of $\partial \Omega$ in the direction out of $\Omega$. If $H\geq 0$,  $(\partial \Omega, \gamma, H)$ is Bartnik data. If $(\Omega, g_{\Omega})$ is glued to an AF extension $(M,g)$ of this Bartnik data along their boundaries to produce a Lipschitz Riemannian manifold, then the condition $H \geq H_{\partial M}$ corresponds to the scalar curvature being distributionally nonnegative along the interface \cites{Miao_corners, ST}.
\end{remark}

The primary motivation for considering the Bartnik mass is the quasi-local mass problem in general relativity. This problem asks for a meaningful definition of ``mass'' to be assigned to a compact 3-dimensional region lying in a spacelike hypersurface in a $(3+1)$-dimensional Lorentzian spacetime. The dominant energy condition on the spacetime together with a time-symmetric (i.e., totally geodesic) assumption on the spacelike hypersurface, is equivalent to nonnegative scalar curvature on the hypersurface. Generally quasi-local mass quantities (of which many have been proposed) depend not on the full geometry of the compact region, but only the Bartnik boundary data of the boundary.

Bartnik originally defined a quasi-local mass of $(\Sigma, \gamma, H)$ by considering all AF extensions of the data with nonnegative scalar curvature, subject to a ``no-horizons'' condition, and taking the infimum of their ADM masses. This is denoted $m_B(\Sigma, \gamma, H)$. 
The no-horizons condition means that the AF extensions cannot contain a minimal surface enclosing the boundary; this was modified by Bray to require instead the boundary not be enclosed by a surface of less area. (Many subtle variants of these conditions have appeared in the literature --- see, for example \cites{McC,Jau, McC2} for discussion.)

\medskip

Inspired by Bartnik's mass, we define a related quantity,  based on the capacity instead of the ADM mass. Specifically, we will maximize the capacity among admissible AF extensions (of fixed Bartnik data). We first recall that the \emph{capacity} of the boundary $\partial M$ in an AF 3-manifold $(M,g)$ is:
$$\capac_g(\partial M) =\inf_{\psi \text{ Lipschitz}} \left\{ \frac{1}{4\pi} \int_M |\nabla \psi|^2_g dV_g \; :\; \psi|_{\partial M}=0, \quad \psi \to 1 \text{ at infinity} \right\}.$$
The infimum is realized uniquely by the $g$-harmonic function $\varphi$ that vanishes on $\partial M$ and tends to 1 at infinity; $\varphi$ is called the \emph{capacitary potential}. Existence and decay properties of $\varphi$ are well-known; we elaborate on these in Lemma \ref{prop_expansion}. From the divergence theorem,
\begin{equation}
\label{eqn_cap}
\capac_g(\partial M) = \frac{1}{4\pi} \int_M |\nabla \varphi |^2_g dV_g = \frac{1}{4\pi} \int_{\partial M} \partial_{\nu_g}(\varphi) dA_g>0,
\end{equation}
where $\nu_g$ is the unit normal to $\partial M$ pointing into $M$.

Finally, we can give our main definition:

\begin{definition}
\label{def_C}
For Bartnik data  $(\Sigma, \gamma, H)$, the \emph{maximum capacity} is:
$$\C(\Sigma, \gamma, H) = \sup_{(M,g)} \left\{ \capac_g(\partial M) \; : \; (M,g) \text{ is an admissible AF extension of } (\Sigma, \gamma, H) \right\}.$$
\end{definition}
If no admissible AF extension exists, then $\C(\Sigma, \gamma, H)$ is taken to be $-\infty$. If such an extension does exist, then clearly $\C(\Sigma, \gamma, H) > 0$, but it is not obvious that $\C(\Sigma, \gamma, H)$ is finite.

Note that, in contrast to the Bartnik mass, in the definition of $\C$ no restriction is placed comparable to ``no horizons'' or ``area-outer-minimizing'' conditions.

 \begin{remark}
Studying the \emph{minimum} value of the capacity would be uninteresting. If an admissible AF extension exists, then the infimum of the capacities among such extensions would always be zero. This is because it is possible to ``hide'' the boundary behind an arbitrarily small neck, which allows the capacity to be arbitrarily small. This is the related to the phenomenon of the Bartnik mass being trivially zero without the no-horizons condition (or some variant thereof), as originally pointed out by Bartnik \cite{Ba1}. One could prohibit small necks and/or require the boundary to be outward-minimizing, but the infimum of capacity would still be zero due to ``long necks.'' We demonstrate this explicitly in section \ref{app_minimizing} in the appendix.
\end{remark}

\begin{remark}
Nonnegativity of scalar curvature is essential in the formulation of Definition \ref{def_C}. Without that restriction, the maximum value of the capacity would always be $+\infty$ (if not $-\infty$). Likewise, $\C$ would also be $\pm\infty$ if the boundary condition on mean curvature were to be omitted in the definition of admissible extension. We prove these claims  in Proposition \ref{prop_hypotheses} in the appendix.
\end{remark}

This paper is motivated by Bartnik's well-known question of whether minimal mass extensions exist. Here we ask: is there an admissible AF extension $(M,g)$ of prescribed Bartnik data that attains the maximum capacity as in Definition \ref{def_C}? We call such a metric a \emph{maximal capacity extension}. We point out that while both the minimal mass and maximal capacity extension problems are evidently quite challenging, there are some indications that the latter could be more tractable. For one, as pointed out above, there is no need for a ``no-horizons'' type condition in the definition of $\C$. Additionally, ``long necks'' rule themselves automatically by being far from optimal from the point of view of maximizing capacity. In other words, some of the challenges associated with the Bartnik mass minimization problem are rendered moot in the capacity  problem. While a general understanding of both the minimal mass and maximal capacity problems remains out of reach, we make some progress on the latter problem in this paper. 

Our main result, Theorem \ref{thm_main}, identifies the variational condition satisfied by a maximal capacity extension. (It is well-known that minimal mass extensions are AF \emph{static metrics} --- see \cites{Cor,AJ} --- as originally conjectured by Bartnik.)
This leads to the idea of what we call a ``harmonic-static metric,'' a generalization of the well-known concept of a static metric. Recall that a static metric $g$ by definition admits a solution $u \neq 0$ to
\begin{equation}
\label{static}
\Hess(u) - (\Delta u)g -u\Ric = 0,
\end{equation}
where the Hessian, Laplacian, and Ricci tensor are with respect to $g$. 
In this paper we define a \emph{harmonic-static metric} to be any metric $g$ on a 3-manifold $M$ that admits a solution $u$ to
\begin{equation}
\label{HS}
\Hess(u) - (\Delta u)g -u\Ric = -d\varphi \otimes d\varphi +\frac{1}{2} |\nabla \varphi|^2 g
\end{equation}
where $\varphi$ is a fixed $g$-harmonic function on $M$. Theorem \ref{thm_main} states that maximal capacity extensions are harmonic-static. In that case, $\varphi$ is the capacitary potential, but we will also consider condition \eqref{HS} for general harmonic functions. Metrics satisfying \eqref{HS} are shown to have constant scalar curvature (and therefore zero scalar curvature in the AF case) --- see Proposition \ref{prop_CSC}. One interpretation of \eqref{HS} is that the homogeneous static equation \eqref{static} has been modified with an inhomogeneous ``stress-energy tensor'' term --- see remark \ref{rmk_stress}.

\begin{outline} First, in section \ref{sec_finite}, we demonstrate that $\C(\Sigma, \gamma,H)$ is nontrivial (i.e., not $\pm \infty$) in a wide range of cases. The technical setup needed for the rest of the paper is provided in section \ref{sec_technical}, along with some lemmas.
Section \ref{sec_results} works out the gradient of the capacity functional and includes the statement and proof of the main theorem: the variational characterization of maximal capacity extensions as harmonic-static metrics (Theorem \ref{thm_critical} is the technical version; Theorem \ref{thm_main} is the main version that is easier to state). We study the harmonic-static condition in a local sense in section \ref{sec_local}, proving such metrics have constant scalar curvature. We also give a local characterization of harmonic-static metrics based on a localization of the maximum capacity problem. Section \ref{sec_examples} gives a few examples of harmonic-static metrics. Additional discussion, including viewing the maximum capacity $\mathcal{C}$ as a ``quasi-local capacity'' appears in section \ref{sec_discussion}, along with some open problems. In the appendix we show the necessity of nonnegative scalar curvature for the maximum capacity problem, and, as an application of the gradient of the capacity calculation, give another proof of the formula for the variation of capacity for a flowing surface in a fixed manifold.
\end{outline}

\section{Finiteness results and bounds}
\label{sec_finite}

We first demonstrate that $\C(\Sigma, \gamma, H)$ is nontrivial by giving an upper bound in terms of Bartnik data. This will come directly from a result of Bray and Miao \cite{BM}.
\begin{prop}
\label{prop_C_upper_bound}
Let $(\Sigma, \gamma, H)$ be Bartnik data. Then 
\begin{equation}
\label{eqn_upper_bound}
\C(\Sigma, \gamma, H) \leq  \sqrt{\frac{|\Sigma|_\gamma}{16\pi}}\left(1+\sqrt{\frac{ 1}{16\pi}\int_{\partial M} H^2 dA } \right),
\end{equation}
where $|\Sigma|_{\gamma}$ is the area of $\Sigma$ with respect to $\gamma$. In particular, $\C(\Sigma, \gamma, H) < \infty$.
\end{prop}
\begin{proof}
First, if no admissible extension of $(\Sigma, \gamma, H)$ exists, then $\C(\Sigma, \gamma, H) = -\infty$. Now, if $(M,g)$ is an AF extension of $(\Sigma, \gamma, H)$ with nonnegative scalar curvature, then Bray and Miao showed in \cite[Theorem 1]{BM} that
$$\capac_g(\partial M) \leq  \sqrt{\frac{|\Sigma|_\gamma}{16\pi}}\left(1+\sqrt{\frac{ 1}{16\pi}\int_{\partial M} H_{\partial M}^2 dA } \right).$$
Since $H_{\partial M} \leq H$ for an admissible extension, the (finite) upper bound \eqref{eqn_upper_bound} for $\C$ follows.
\end{proof}

\begin{remark}
It would be desirable to relax the assumption in our definition of AF extension that the manifold be diffeomorphic to the complement of a bounded domain. In that case, Bray and Miao's result would not apply due to a topological restriction in their argument. Alternatively, if we restricted to AF extensions such that $H_{\partial M} < 4 |\nabla \varphi|$ for the capacitary potential $\varphi$ of the extension, then the finiteness result in Proposition \ref{prop_C_upper_bound}, along with an upper bound on $\mathcal{C}$, would follow from the work of Mantoulidis--Miao--Tam \cite{MMT}. This includes the ``horizon'' case in which $H_{\partial M}=0$. 
\end{remark}

Using known results on the existence of AF extensions, we can infer that $\C > -\infty$ for a broad class of Bartnik data.
\begin{prop}
\label{prop_ext_exist}
Let $(\Sigma, \gamma, H)$ be Bartnik data, where $\Sigma$ is a topological 2-sphere. Suppose $\gamma$ has positive Gauss curvature (or, more generally, that the first eigenvalue of the operator $-\Delta_\gamma + K$ is positive). Then $\C(\Sigma, \gamma, H) > -\infty$, so in fact  $\C(\Sigma, \gamma, H) > 0$.
\end{prop}
\begin{proof}
In \cite{MS} Mantoulidis and Schoen construct AF 3-manifolds, diffeomorphic to a ball complement, with nonnegative scalar curvature whose boundary is a minimal surface and can be arranged to be isometric to an arbitrary metric on $S^2$ with $\lambda_1(-\Delta_\gamma+K)>0$. Such a manifold is an admissible AF extension of $(\Sigma,\gamma,H)$ (since $H \geq 0$ by assumption). Thus $\C(\Sigma, \gamma, H) > -\infty$. Since the capacity is always positive in any AF extension, $\C(\Sigma, \gamma, H) > 0$.
\end{proof}

In the rotationally symmetric case, the value of $\C$ is explicitly computable, and maximal capacity extensions can be characterized: 
\begin{prop}
\label{prop_round}
Suppose $(\Sigma, \gamma)$ is a round 2-sphere and $H \geq 0$ is a constant. Then
\begin{equation}
\label{eqn_C_rot_sym}
\C(\Sigma, \gamma, H) = \sqrt{\frac{|\Sigma|_\gamma}{16\pi}}\left(1+H\sqrt{\frac{ |\Sigma|_{\gamma}}{16\pi} } \right).
\end{equation}
Moreover, there exists a unique maximal capacity extension of $(\Sigma,\gamma,H)$ (up to isometry), given by a subset of a Schwarzschild manifold.
\end{prop}
\begin{proof}
Given Bartnik data $(\Sigma, \gamma, H)$ with $\gamma$ round and $H \geq 0$ constant, it is well known and straightforward to verify that there exists a unique $m \in \R$ and $r \geq \frac{m}{2}$ (if $m > 0$) or $r > \frac{|m|}{2}$ (if $m \leq 0$) such that in the Schwarzschild metric of mass $m$, i.e.,
$$g_{ij} = \left(1+ \frac{m}{2|x|}\right)^4 \delta_{ij},$$
the  $|x|=r$ coordinate sphere $S_r$ is isometric to $\gamma$ and has mean curvature equal to $H$. In particular, the region $M_r$ corresponding to $|x| \geq r$ is an admissible extension of $(\Sigma, \gamma,H)$. Noting  the function
$\frac{ 1-\frac{r}{|x|}} {1+\frac{m}{2|x|}}$
is the capacitary potential of $S_r=\partial M_r$ in $M_r$, it is straightforward to check that the capacity of $S_r$ in $M_r$ equals $r+\frac{m}{2}$.

On the other hand, by direct computation, one can see that the right-hand side of \eqref{eqn_C_rot_sym} also equals $r+\frac{m}{2}$ and hence equals the capacity of $S_r$ in $M_r$. This shows
$$\C(\Sigma, \gamma, H) \geq \sqrt{\frac{|\Sigma|_\gamma}{16\pi}}\left(1+H\sqrt{\frac{ |\Sigma|_{\gamma}}{16\pi} } \right).$$
However,  \cite[Theorem 1]{BM} gives the reverse inequality, so the rigidity case of this theorem applies. In particular, any admissible extension of $(\Sigma, \gamma, H)$ with capacity equal to $\C(\Sigma,\gamma,H)$ is isometric to $(M_r,g)$.
\end{proof}

\section{Technical setup}
\label{sec_technical}
Fix Bartnik data $(\Sigma, \gamma, H)$. Let $\Omega \subset \R^3$ be a bounded domain whose boundary is diffeomorphic to $\Sigma$. Since any extension of $(\Sigma, \gamma, H)$ is diffeomorphic to $\R^3 \setminus \Omega$ by definition, we will simply fix $M = \R^3 \setminus \Omega$ and consider AF metrics $g$ on $M$. Note $M$ inherits the standard global coordinate chart $(x^i)$ from $\R^3$. We fix a smooth function $\sigma \geq 1$ on $M$ with $\sigma = |x|$ outside some compact set. Let $\sigma(x,y) = \min(\sigma(x), \sigma(y))$. We now recall the definition of weighted H\"older spaces, closely following \cite{Cor2}. For $\alpha \in (0,1)$, $\beta>0$, and a function $f$ on $M$, let
$$[f]_{\alpha, - \beta} = \sup_{x, y \in M, x \neq y} \sigma(x,y)^{\alpha+\beta} \frac{|f(x)-f(y)|}{|x-y|^\alpha}.$$
The weighted H\"older space $C^{k,\alpha}_{-\beta}(M)$, a Banach space, consists of all functions $f \in C^{k,\alpha}(M)$ such that the norm
$$\|f\|_{C^{k,\alpha}_{-\beta}(M)} = \sum_{|\lambda| \leq k} \|\sigma^{\beta + |\lambda|} \partial_\lambda f \|_{L^\infty(M)} + \sum_{|\lambda|=k} [\partial_\lambda f]_{\alpha, -\beta-k}$$
is finite, where the sums are taken over multi-indices $\lambda$.

Now, fix an integer $k \geq 2$, $\alpha \in (0,1)$, and $\tau \in \left(\frac{1}{2}, 1\right)$. We consider the space $\M^{k,\alpha}_{-\tau}(M)$ of Riemannian metrics $g$ on $M$ whose coefficients satisfy
$$g_{ij}-\delta_{ij} \in C^{k,\alpha}_{-\tau}(M),$$
i.e. metrics $g$ that decay appropriately to $\delta_{ij}$ at infinity in the fixed coordinate chart. We will sometimes omit the ``$-\tau$'' from the notation for simplicity.
Since $C^{k,\alpha}_{-\tau}(M)$ is a Banach space, $\M^{k,\alpha}(M)$ is a smooth Banach manifold. Later, we will show in Lemma \ref{lemma_regularity} that using elements of $\M^{k,\alpha}(M)$ as test functions for $\C$ is equivalent to using general AF extensions as in the definition of $\C$ from the introduction.

We gather some results on weighted elliptic estimates as stated in \cite{Cor2}, Proposition 2.4:

\begin{prop}
\label{prop_estimates}
Let $\beta >0$. Suppose $g \in \M^{k-1,\alpha}_{-\tau}(M)$. If $w \in C^0_{-\beta}(M)$ and $\Delta_g w \in C^{k-2,\alpha}_{-\beta-2}(M)$, then 
$w \in C^{k,\alpha}_{-\beta}(M)$. Moreover, there exists $C>0$ such that for all $w \in C^{k,\alpha}_{-\beta}(M)$ with $w=0$ on $\partial M$,
\begin{equation}
\label{ell_estimate1}
\|w\|_{C^{k,\alpha}_{-\beta}} \leq C \left( \|\Delta_g w\|_{C^{k-2,\alpha}_{-\beta-2}} + \|w\|_{C^{0,\alpha}_{-\beta}}\right).
\end{equation}
Also for all $w \in C^{k,\alpha}_{-\tau}(M)$ with $w=0$ on $\partial M$,
\begin{equation}
\label{ell_estimate2}
\|w\|_{C^{k,\alpha}_{-\tau}} \leq C  \|\Delta_g w\|_{C^{k-2,\alpha}_{-\tau-2}}.
\end{equation}
Finally, $\Delta_g : C^{k,\alpha}_{-\tau}(M) \to C^{k-2,\alpha}_{-\tau-2}(M)$ restricts to an isomorphism from the subspace of the domain consisting of functions that vanish on $\partial M$.
\end{prop}

We recall the existence, regularity, and asymptotics of the capacitary potential. 
\begin{lemma}
\label{prop_expansion}
If $g \in \M^{k,\alpha}(M)$, then the capacitary potential $\varphi \in C^{k+1,\alpha}_{loc}(M)$  of $(M,g)$ exists and satisfies
$$\varphi(x) = 1-\frac{\capac_g(\partial M)}{|x|} + E(x),$$
where the error term $E(x)$ belongs to $C^{k+1,\alpha}_{-\gamma}(M)$ for a constant $\gamma > 1$. In particular, $1-\varphi \in C^{k+1,\alpha}_{-\tau}(M)$.
\end{lemma}
\begin{proof}
Existence is well-known, and the local regularity follows from standard elliptic theory (or from \eqref{ell_estimate2}). The expansion follows from Proposition 2.6 of \cite{Cor2}, for example. We also refer the reader to   \cite[Lemma A.2]{MMT} for a similar asymptotic expression for the potential function.
\end{proof}

Next, we consider a family of metrics and the corresponding family of potentials.
\begin{lemma}
\label{lemma_path}
Let $t \mapsto g_t$ be a differentiable path in $\M^{k,\alpha}(M)$ defined on $(-\epsilon, \epsilon)$, and let $\varphi_t$ be the corresponding capacitary potentials. Then $t\mapsto 1-\varphi_t$ is a differentiable path in $C^{k+1,\alpha}_{-\tau}(M)$. Moreover, 
there exists $A \in \R$ such that
$$\left.\frac{d}{dt} \varphi_t \right|_{t=0} = \frac{A}{|x|} + E(x),$$
where the error term $E(x)$ belongs to $C^{k+1,\alpha}_{-\gamma}(M)$ for a constant $\gamma > 1$. In fact, 
$$A= - \frac{d}{dt} \left.\capac_{g_t} (\partial M) \right|_{t=0}.$$
\end{lemma}
\begin{proof}
The proof of differentiability is a straightforward application of the implicit function theorem for Banach manifolds, which we carry out below. In this proof, we will subtract our usual capacitary potentials from 1, so that they equal 1 on $\partial M$ and decay to 0 at infinity. Let $^1C^{k+1,\alpha}_{-\tau}(M)$ denote the closed Banach submanifold of $C^{k+1,\alpha}_{-\tau}(M)$ consisting of functions that restrict to 1 on $\partial M$, whose tangent space may be naturally identified with $^0C^{k+1,\alpha}_{-\tau}(M)$, the closed Banach subspace of $C^{k+1,\alpha}_{-\tau}(M)$ consisting of functions that restrict to 0 on $\partial M$.

Let $F: \M^{k,\alpha}(M) \times \! \,^1C^{k+1,\alpha}_{-\tau}(M) \to C^{k-1,\alpha}_{-\tau-2}(M)$ be the map:
$$F(g,\psi) = \Delta_g (\psi).$$

We recall that for a differentiable path of Riemannian metrics $g_t$ with $\frac{dg_t}{dt}|_{t=0} = h$, the variation of the Laplacian for a fixed function $f$ is:
\begin{equation}
\label{eqn_var_lap}
\left.\frac{\partial}{\partial t} \Delta_{g_t} f\right|_{t=0} = - \langle \Hess f, h \rangle - \langle \Div(h) - \frac{1}{2} d(\tr h), df \rangle,
\end{equation}
where the pairing, Hessian, divergence, and trace are with respect to $g_0$. This formula appears, for example, as Exercise 2.31 in \cite{CLN}.
From this formula, it is straightforward to verify that $F$ is continuously Fr\'echet differentiable\footnote{We recall that to show a function is continously Fr\'echet differentiable, it suffices to verify that the directional (Gateaux) derivative exists and is continuous on a neighborhood; see, for example, Proposition 3.2.15 of \cite{DM}.}.

Let $g_0 \in \M^{k,\alpha}(M)$, and let $\varphi_0$ be its capacitary potential. By Proposition \ref{prop_expansion}, $1-\varphi_0 \in C^{k+1,\alpha}_{-\tau}(M)$, so $F(g_0, 1-\varphi_0)=0$.  Now consider the derivative map $\psi \mapsto DF_{(g_0,\varphi_0)}(0,\psi)$ from $^0C^{k+1,\alpha}_{-\tau}(M)$ to $C^{k-1,\alpha}_{-\tau-2}(M)$, i.e., $DF_{(g_0,\varphi_0)}(0,\psi) = \Delta_{g_0} \psi$. By Proposition \ref{prop_estimates}, this map is an  isomorphism.

By the implicit function theorem, there exists a continuously Fr\'echet differentiable map $w$ from 
a neighborhood of $g_0$ in $\M^{k,\alpha}(M)$ to a neighborhood of $\varphi_0$ in $^1C^{k+1,\alpha}_{-\tau}(M)$ such that $\Delta_g(w(g))=0$ for all $g$ in the neighborhood. In particular, 
$1-w(g)$ is the capacitary potential  for $g$. We conclude that if $g_t$ is a differentiable path of Riemannian  metrics in $\M^{k,\alpha}(M)$, then $w(g_t)=1-\varphi_t$ is a differentiable path in $C^{k+1,\alpha}_{-\tau}(M)$. It follows that $\psi:=\left.\frac{d}{dt} \varphi_t \right|_{t=0} \in C^{k+1,\alpha}_{-\tau}(M)$. 

To obtain $O(r^{-1})$ decay of $\psi$ from the known $O(r^{-\tau})$ decay, we compute the $g_0$ Laplacian of $\psi$. Differentiating $\Delta_{g_t} \varphi_t = 0$ and using \eqref{eqn_var_lap}, we obtain
$$\Delta_{g_0} \psi = \langle \Hess \varphi, h \rangle + \langle \Div(h) - \frac{1}{2} d(\tr h), d\varphi\rangle.$$
Since $1-\varphi \in C^{k+1,\alpha}_{-\tau}(M)$ and $ h_{ij}  \in C^{k,\alpha}_{-\tau}(M)$, we see that the right-hand side belongs to $C^{k-1,\alpha}_{-2\tau -2}(M)$. Since $\tau > \frac{1}{2}$, $\beta := 2\tau + 2 > 3$. Then by Proposition 2.6 of \cite{Cor2}, we have
$$\psi(x) = \frac{A}{|x|} + O_{k+1,\alpha}(|x|^{-\gamma})$$
for some constant $A \in \R$ and $\gamma > 1$. The expression for $A$ follows from differentiating the expansion at infinity for $\varphi_t$.
\end{proof}

\medskip

We next point out that using metric tensors in the fixed coordinate chart (as in the definition of $\M^{k,\alpha}(M)$) is essentially equivalent to using abstract AF metrics on $M$ (as in Definition \ref{def_AF}). The following remark explains the basic idea, with the full proof of  equivalence given in Lemma \ref{lemma_regularity} in the appendix.

\begin{remark}
\label{rmk_diffeo}
If $(N,g)$ is an AF extension of $(\Sigma, \gamma, H)$ as in Definition \ref{def_extension}, we claim $g$ is isometric to a Riemannian metric $h$ with $h_{ij} - \delta_{ij}$ decaying to zero as $|x| \to \infty$ (i.e., in the fixed coordinate chart).  This follows from the observation that a diffeomorphism $\Phi: N \setminus K \to \mathbb{R}^3 \setminus B(0,r)$ as in the definition of asymptotically flat, with $\overline\Omega$ contained in the interior of $B(0,r)$, can be smoothly extended to a diffeomorphism $\hat \Phi: N \to \R^3 \setminus \Omega = M$ (which follows from standard results in differential topology). Then $\hat \Phi_* g$ is the desired metric $h$. Moreover, if $(N,g)$ is admissible, so is $(M, \hat \Phi_* g)$.  In particular, for the problem of maximizing capacity on admissible AF extensions, it suffices work in the fixed chart. Lemma \ref{lemma_regularity} shows that, in the definition of $\C$, the specifics of the regularity and decay (i.e., Definition \ref{def_AF} vs. $\M^{k,\alpha}(M)$) are immaterial.
\end{remark}

To conclude this section, we recall that the linearization of the scalar curvature map $s:\M^{k,\alpha}(M) \to C^{k-2,\alpha}_{-\tau -2}(M)$ (for $k \geq 2$), $g \mapsto R_g$, at a Riemannian metric $g$ is given by:
\begin{equation}
\label{eqn_L_g}
L_g(h) = -\Delta_g(\tr_g(h)) + \Div(\Div(h)) - g(\Ric_g,h). 
\end{equation}

The formal $L^2$-adjoint of $L_g$ is given by
$$L_g^*u = \Hess_g(u) -(\Delta_g u) g - u\Ric_g.$$
We refer the reader to \cite{Cor} for a thorough account of $L_g$ and $L_g^*$.

\section{Variational characterization of maximum capacity extensions}
\label{sec_results}

In this section we study the geometry of maximum capacity extensions, i.e., admissible AF extensions of some Bartnik data that achieve the maximum possible value of capacity.

We make a preliminary simple but key observation, analogous to the observation in \cite{Cor} that Bartnik mass-minimizers must be scalar flat. A different argument from that in \cite[Theorem 8]{Cor} is necessary, as the global deformation in that case decreases the capacity (and therefore does not produce a contradiction).
\begin{prop}
\label{prop_ZSC}
If $(M,g)$ is a maximal capacity extension of some Bartnik data $(\Sigma, \gamma, H)$, then $g$ has identically zero scalar curvature.
\end{prop}
\begin{proof}
Let $(M,g)$ be a maximal capacity extension of $(\Sigma, \gamma, H)$; in particular, $(M,g)$ is admissible.
Suppose $R_g(p) > 0$ for some $p \in M$; without loss of generality, we may assume $p$ lies in the interior of $M$. Then there exists an open set $U$ compactly contained in the interior of $M$ and $\epsilon > 0$ such that $R_g \geq \epsilon$ on $\overline{U}$. Consider a smooth path of  conformal metrics $g_t = (1+t\rho)^4 g,$ where $\rho \geq 0$ is a smooth bump function that is not identically zero and is supported in $U$. Note that $g_t$ is asymptotically flat and agrees with $g$ outside $U$, and in particular near $\partial M$.
Moreover, for $t > 0$ sufficiently small, $R_{g_t} \geq 0$, so that $g_t$ is an admissible extension of  $(\Sigma, \gamma, H)$.

We claim that the capacity of $\Sigma$ in $(M,g_t)$  strictly exceeds $\C(\Sigma, \gamma,H)$, which will be a contradiction and complete the proof. For $t > 0$, let $\varphi_t$ be the capacitary potential of $\Sigma$ in $(M,g_t)$. Then for $t$ small:
\begin{align*}
\capac_{g_t}(\partial M) &= \frac{1}{4\pi} \int_{M} |d \varphi_t|^2_{g_t} dV_{g_t}\\
&=  \frac{1}{4\pi} \int_{M} |d \varphi_t|^2_{g} (1+t\rho)^2 dV_{g}\\
&>  \frac{1}{4\pi} \int_{M} |d \varphi_t|^2_{g} dV_{g}\\
&\geq \capac_g(\partial M)\\
&= \C(\Sigma, \gamma,H),
\end{align*}
where strict inequality occurs because $\varphi_t$, being $g_t$-harmonic and non-constant, cannot have its gradient vanish on the open set $\{\rho >0\}$.
\end{proof}

It is known that an AF admissible extension minimizing the Bartnik mass (thus satisfying the $H \geq H_{\partial M}$ boundary condition) actually satisfies $H = H_{\partial M}$ (see \cite{Miao_boundary} and \cite{AJ}). By the same intuition given in \cite{Miao_boundary} (that any strict jump in mean curvature ought be exchangeable for positive scalar curvature in the interior), it is reasonable to expect the same conclusion to hold for maximal capacity extensions. We show a partial result in this direction:
\begin{prop}
Suppose $(M,g)$ is a maximal capacity extension of some Bartnik data $(\Sigma, \gamma, H)$. Then $H>H_{\partial M}$ is impossible, i.e. $H_{\partial M} = H$ at at least one point.
\end{prop}
\begin{proof}
Suppose $(M,g)$ is an admissible extension of Bartnik data $(\Sigma, \gamma, H)$ with $H_{\partial M} < H$.
Let $\varphi$ be the capacitary potential of $(M,g)$. For a parameter $c>1$, let $u_c = 1 + (c-1)\varphi$. Thus, $\Delta_g u_c = 0$, $u_c=1$ on $\partial M$, and $u_c \to c$ at infinity. By the maximum principle, $u_c>1$ except on $\partial M$. The conformal metric $ g_c = u_c^4 g$ is AF with nonnegative scalar curvature, and its boundary is isometric to $(\Sigma, \gamma)$.   Moreover, since $u_c > 1$ in the interior of $M$, the capacity of $\partial M$ with respect to $g_c$ is strictly larger than with respect to $g$.  The mean curvature of $\partial M$ with respect to $g_c$ is given by $H_{\partial M} + 4\partial_{\nu}(u_c) = H_{\partial M} + 4(c-1)\partial_{\nu}(\varphi)$. In particular, for $c>1$ sufficiently close to 1, this mean curvature is still less than $H$, and hence $(M,g_c)$ is an admissible extension. It follows that $(M,g)$ cannot be a maximal capacity extension.
\end{proof}

\medskip

At this point, we reformulate the definition of the maximum capacity $\C$ used in sections \ref{sec_intro} and \ref{sec_finite} so as to be phrased in terms of weighted H\"older-regular AF metrics. This is for analytical convenience. For the rest of this section, we use the definition (where we reiterate the assumption $M=\R^3 \setminus \Omega$ from section \ref{sec_technical}):
\begin{equation}
\label{eqn_C_tilde}
\C(\Sigma, \gamma, H) = \sup_{(g)} \left\{ \capac_g(\partial M) \; : \; g \in \M^{k,\alpha}_{-\tau}(M), R_g \geq 0, g|_\Sigma = \gamma, \text{ and } H \geq H_{\partial M} \right\}.
\end{equation}
In the appendix, we show that the old and new definitions of $\C$ are equivalent (Lemma \ref{lemma_regularity}).

Our goal in this section is to derive the variational property satisfied by maximum capacity extensions, similar to how Bartnik mass minimizers are well-known to be static vacuum.
We approach this using a Lagrange multiplier argument, as in \cite{AJ}, following a heuristic suggested by Bartnik \cite{Ba_phase}. That is, we seek to find critical points of the capacity on AF metrics that live on the constraint set of scalar-flat metrics (since maximizers are necessarily scalar-flat by Proposition \ref{prop_ZSC}). A necessary ingredient in this is a general formula for how the capacity behaves under an infinitesimal deformation of the metric:

\begin{prop}
\label{prop_Dcap}
The capacity functional $\F: \M^{k,\alpha}(M) \to \R$, $\F(g) = 4\pi \capac_g(\partial M)$, is continuously Fr\'echet differentiable.  
The $L^2$-gradient of $\F$ at $g \in \M^{k,\alpha}(M)$ is given by 
$$-d\varphi \otimes d\varphi + \frac{1}{2} |\nabla \varphi|^2 g,$$ 
where $\varphi$ is the capacitary potential for $(M,g)$.
\end{prop}

\begin{remark}
\label{rmk_stress}
The tensor $-d\varphi \otimes d\varphi + \frac{1}{2} |\nabla \varphi|^2 g$ equals the stress-energy tensor of the function $\varphi: M \to \R$ as defined by Baird and Eells within the context of harmonic maps \cite{BE}. Indeed, for a map $f:(M,g) \to (N,h)$ between Riemannian manifolds, the stress-energy tensor is $S_f = \frac{1}{2} |df|^2g - f^*h$. This reduces as claimed when $(N,h)$ is the real line and $f=\varphi$. The fact that $\varphi$ is a harmonic map (i.e., a harmonic function) implies that $S_\varphi$ is divergence-free, a fact we will use in Proposition \ref{prop_CSC}.
\end{remark}

\begin{proof}
From the proof of Lemma \ref{lemma_path}, the assignment $g \mapsto 1-\varphi_g$, where $\varphi_g$ is the capacitary potential with respect to $g$, is continuously Fr\'echet differentiable. Using \eqref{eqn_cap}, one can see that $\mathcal{F}$ itself is continuously Fr\'echet differentiable.

We now determine the derivative of $\F$. Let $\{g_t\}$ be a differentiable path of AF metrics in $\M^{k,\alpha}(M)$ with $g_0=g$, and let $h=\frac{\partial}{\partial t}g_t|_{t=0}$.
Let $\varphi_t$ be the corresponding capacitary potentials, referring to $\varphi_0$ as $\varphi$. From Lemma \ref{lemma_path}, $1-\varphi_t$ is a differentiable path of functions in $C^{k+1,\alpha}_{-\tau}(M)$.  Let $\psi = \frac{\partial}{\partial t}\varphi_t|_{t=0}$, and note $\psi|_{\partial M} = 0$.
We have (letting $\nabla$ refer to the gradient with respect to $g$):
\begin{align*}
\frac{d}{dt}\left.4\pi \capac_{g_t}(\partial M) \right|_{t=0} &= \frac{d}{dt}\left. \int_M |d\varphi_t|^2_{g_t} dV_{g_t} \right|_{t=0}\\
&=\int_M \left(2g(\nabla \varphi, \nabla \psi) - h(\nabla \varphi, \nabla \varphi) + \frac{1}{2}|\nabla \varphi|_g^2 \tr_g(h) \right) dV_{g}.
\end{align*}
We claim the first term (which is clearly integrable) has vanishing integral. Since $\psi=0$ on $\partial M$ and $\Delta \varphi = 0$, by the divergence theorem we have
\begin{align*}
\int_M   g(\nabla \varphi, \nabla \psi) dV_g&= \int_{M} \Div_g(\psi \nabla \varphi) dV_g\\
 &= \lim_{r \to \infty} \int_{S_r} \psi \partial_{\nu}(\varphi) dA.
\end{align*}
Since $ \lim_{r \to \infty} \int_{S_r}  \partial_{\nu}(\varphi) dA = \int_{\partial M}  \partial_{\nu}(\varphi) dA$ is finite and $\psi = O(|x|^{-1})$, the above limit equals zero.

Thus we have:
$$\frac{d}{dt}\left.4\pi\capac_{g_t}(\partial M) \right|_{t=0} = \int_M \left\langle -d\varphi \otimes d\varphi + \frac{1}{2}|\nabla \varphi|^2 g, h\right\rangle dV_g$$
where $\langle, \rangle$ is the inner product with respect to $g$. 
\end{proof}

In Proposition \ref{prop_cap_varying_boundary} in the Appendix, we point out that Proposition \ref{prop_Dcap} can be used to derive the variation of capacity of a smooth flow of hypersurfaces in a fixed AF manifold.

Recalling $\M^{k,\alpha}(M)$ is the Banach manifold of AF metrics on $M$, we let $\M^{k,\alpha}_{(\gamma,H)}(M)$ be the closed Banach submanifold consisting of those elements that induce Bartnik boundary data $(\gamma, H)$.
We let $\Z^{k,\alpha}_{(\gamma,H)}(M)$ the subset whose scalar curvature is zero on $M$. In \cite{AJ} it was proved that  $\Z^{k,\alpha}_{(\gamma,H)}(M)$ is a smooth Banach submanifold\footnote{This was subsequently generalized by Z. An to establish that initial data sets with fixed spacetime Bartnik boundary data and satisfying the Einstein constraint equations for fixed energy and momentum densities also form a smooth Banach manifold \cite{An}.} of $\M^{k,\alpha}_{(\gamma,H)}(M)$. 

As above, let $\F: \M^{k,\alpha}(M) \to \R$ be the functional $\F(g) = 4\pi\capac_g(\partial M)$.

\begin{thm}
\label{thm_critical}
Suppose $g \in \Z^{k,\alpha}_{(\gamma,H)}(M)$ is a critical point for the restriction of $\F$ to $\Z^{k,\alpha}_{(\gamma,H)}(M)$. (In other words, the metric $g$ is a critical point for the capacity when restricted to scalar-flat metrics inducing boundary data $(\gamma, H)$ on $\partial M$.) Then there exists a  function $u \in C^{k+1,\alpha}_{loc}(M)$
 on $M$ such that
\begin{equation}
\label{eqn_HS_AF}
\Hess(u) - (\Delta u)g -u\Ric = -d\varphi \otimes d\varphi +\frac{1}{2} |\nabla \varphi|^2 g,
\end{equation}
where $\varphi$ is the capacitary potential of $(M,g)$, and the Hessian, Laplacian, and Ricci curvature are with respect to $g$. 
If there are two distinct solutions $u$ to \eqref{eqn_HS_AF}, then $g$ is static vacuum.
\end{thm}

Note that using the trace, equation \eqref{eqn_HS_AF} can be rewritten as:
\begin{equation}
\label{eqn_traced_version}
\Hess(u) = u\Ric -d\varphi \otimes d\varphi + \frac{1}{4}|\nabla \varphi|^2 g,
\end{equation}
since $g$ has zero scalar curvature.

\begin{proof}[Proof of Theorem \ref{thm_critical}]
Given the hypothesis, we may assume $g \in \Z^{k,\alpha}_{(\gamma,H)}(M)$ satisfies $D\F_g(h)=0$ for all variations $h$ tangent to $\Z^{k,\alpha}_{(\gamma,H)}(M)$ at $g$ (i.e. $L_g(h)=0$, 
$h^T=0$, and $H_h'=0$, using the notation from \cite{AJ}). Consider the scalar curvature map $s:\M^{k,\alpha}(M) \to C^{k-2,\alpha}_{-\tau -2}(M)$, which is continuously Fr\'echet differentiable. Let $L_{g_0}$ 
denote the derivative of $s$ at $g_0$. Recall from \cite[Proposition 2.4]{AJ} that $L_{g_0}$ is surjective with splitting kernel. Then by a Lagrange multiplier argument 
(see \cite[Theorem 6.3]{Ba_phase}, for example), there exists bounded linear functional $\lambda$ on $C^{k-2,\alpha}_{-\tau-2}(M)$ such that for all variations $h$ of $g$ that are tangent to (the ``unconstrained'' space)
$\M^{k,\alpha}_{(\gamma,H)}(M)$ (i.e., $h^T=0$, and $H_h'=0$), we have
$$d\F_g(h) = \lambda (L_g(h)).$$
From Proposition \ref{prop_Dcap}, this means
\begin{equation}
\label{eqn_lambda}
\int_M \langle -d\varphi \otimes d\varphi + \frac{1}{2}|\nabla \varphi|^2 g, h \rangle dV_g = \lambda(L_g(h)).
\end{equation}

Note that $\lambda$ restricts to a distribution defined on $C_c^\infty(U)$, where $U$ is the interior of $M$. We argue as follows that $\lambda$ has good regularity. Let $f \in C_c^\infty(U)$, and consider variations of the form $h=fg$ in \eqref{eqn_lambda}. Since $g$ has zero scalar curvature, $L_g(fg) = -2 \Delta f$ by \eqref{eqn_L_g}. Thus, we have
$$\lambda(\Delta f) = -\frac{1}{4} \int_M |\nabla \varphi|^2 f dV_g.$$
Now, let $w$ solve the elliptic problem 
\begin{align*}
\begin{cases}
\Delta w = -\frac{1}{4}  |\nabla \varphi|^2\\
w = 0 \text{ on } \partial M\\
w \to 0 \text{ at infinity}.
\end{cases}
\end{align*}
(A solution $w \in C^{k+1,\alpha}_{-\tau}(M)$ exists by Proposition \ref{prop_estimates}.) Let $\mu$ be the functional on $C_c^\infty(U)$ defined by $\mu(f) = \int_M wf$. In particular, $\lambda - \mu$ is a distribution satisfying
$$(\lambda-\mu)(\Delta f) = 0.$$
That is, $\lambda - \mu$ is a weak solution to Laplace's equation on the interior of $M$. By the Weyl Lemma, $\lambda-\mu$ is represented in the interior of $M$ by pairing with harmonic function. It follows that
$$\lambda(f) = \int_M uf$$
where $\Delta u = -\frac{1}{4}|\nabla \varphi|^2$ and $u$ is in $C^{k+1,\alpha}_{loc}$ in the interior of $M$.

We next show $u$ satisfies \eqref{eqn_HS_AF}. Let $h$ be a smooth symmetric 2-tensor on $M$, compactly supported in the interior. From \eqref{eqn_lambda}, we now have
$$\int_M \langle -d\varphi \otimes d\varphi + \frac{1}{2}|\nabla \varphi|^2 g, h \rangle dV_g = \int_M L_g(h) u dV_g = \int_M \langle L_g^*(u), h\rangle dV_g.$$
It follows that
$$L_g^* u = -d\varphi \otimes d\varphi + \frac{1}{2}|\nabla \varphi|^2 g$$
in the interior of $M$, where we recall $L_g^* u = \Hess(u) - (\Delta u)g -u\Ric$. Rewriting this equation using its trace (as done in \eqref{eqn_traced_version}), we deduce that 
$u$ is in $C^{k+1,\alpha}_{loc}(M)$, i.e., is regular up to the boundary.

If $u' \neq u$ solves \eqref{eqn_HS_AF}, then $u-u' \in \ker L_g^*$, i.e., $u-u'$ is a static potential, so $g$ is static.
\end{proof}

\begin{thm}
\label{thm_main}
Suppose $g \in \M^{k,\alpha}(M)$ is a maximal capacity extension of Bartnik data $(\Sigma, \gamma, H)$. Then there exists $u \in C^{k+1,\alpha}_{loc}(M)$
such that
\begin{equation}
\Hess(u) - (\Delta u)g -u\Ric = -d\varphi \otimes d\varphi +\frac{1}{2} |\nabla \varphi|^2 g,
\end{equation}
where $\varphi$ is the capacitary potential of $(M,g)$.
\end{thm}
\begin{proof}
Let $H_0 \leq H$ be the mean curvature of $\partial M$ with respect to $g$. Then $g$ also achieves of the maximum capacity on the smaller set of admissible extensions whose mean curvature equals $H_0$ exactly. From Proposition \ref{prop_ZSC}, $g$ must have zero scalar curvature, i.e., $g \in \Z^{k,\alpha}(M)$. Putting these factors together, we see that $g$ is a critical point of the capacity function $\F$ on $\Z^{k,\alpha}_{(\gamma,H_0)}(M)$. Now, the claim follows from Theorem \ref{thm_critical}.
\end{proof}

\section{Harmonic-static metrics: local analysis and interpretation}
\label{sec_local}

We recall the definition of static Riemannian metrics, following \cite{Cor}.
\begin{definition}
A Riemannian 3-manifold $(M,g)$ is \emph{static} if there exists a function $u$ (a ``static potential''), not identically zero, such that
\begin{equation}
\label{eqn_static}
\Hess(u) - (\Delta u)g -u\Ric = 0,
\end{equation}
and \emph{static vacuum} if, in addition, $\Delta u =0$.
\end{definition}
Recall the left-hand side of \eqref{eqn_static} is $L_g^*u$, the formal $L^2$-adjoint of the linearized scalar curvature operator, so static simply means the kernel of $L_g^*$ is nontrivial. The geometric interpretation of a solution pair $(g,u)$ in \eqref{eqn_static} is that the Lorentzian metric $-u^2dt^2 + g$  (defined where $u \neq 0$) is Einstein (and Ricci-flat in the vacuum case).

Based on the condition \eqref{eqn_HS_AF} identified in Theorem \ref{thm_critical} in the asymptotically flat case, we make the following more general definition.
\begin{definition}
Consider a Riemannian 3-manifold $(M,g)$ equipped with a $g$-harmonic function $\varphi$. Then $(M,g)$ is
\emph{harmonic-static (with respect to $\varphi$)} if there exists a function $u$ (called the potential), not identically zero, such that
\begin{equation}
\label{eqn_HS}
\Hess(u) - (\Delta u)g -u\Ric = -d\varphi \otimes d\varphi +\frac{1}{2} |\nabla \varphi|^2 g.
\end{equation}
\end{definition}
\medskip
A few simple observations are in order:
\begin{enumerate}
\item  If $\varphi$ is constant (which necessarily holds if $M$ is compact for example), then the harmonic-static equation, \eqref{eqn_HS}, reduces to the static equation, \eqref{eqn_static}.
\item One can view \eqref{eqn_HS}, i.e., $L_g^* u =  -d\varphi \otimes d\varphi +\frac{1}{2} |\nabla \varphi|^2 g$, as an inhomogeneous version of the static equation $L_g^*u=0$.
\item The trace of \eqref{eqn_HS} is:
$$-2\Delta u - uR = \frac{1}{2}|\nabla \varphi|^2.$$
\end{enumerate}

\medskip

We recall from \cite{FM} that static metrics have constant scalar curvature (on each connected component). The same is true for harmonic-static metrics:
\begin{prop}
\label{prop_CSC}
Suppose $(M,g)$ is a connected Riemannian manifold that is harmonic-static with respect to some harmonic function $\varphi$.
Then $g$ has constant scalar curvature. In particular, if $(M,g)$ is asymptotically flat, then $g$ has identically zero scalar curvature.
\end{prop}

\begin{proof}
A direct computation (see \cite{FM} or \cite{Cor}) shows that the divergence of $L_g^*u$ is $-\frac{1}{2} u dR_g$. It is straightforward to calculate that the divergence of 
$-d\varphi \otimes d\varphi +\frac{1}{2} |\nabla \varphi|^2 g$ is zero, since $\varphi$ is $g$-harmonic. Thus $u dR_g \equiv 0$ on $M$. If $R_g$ is non constant, then $u \equiv 0$ on the nonempty open 
set $W=\{dR_g \neq 0\}$. Then $ -d\varphi \otimes d\varphi +\frac{1}{2} |\nabla \varphi|^2 g =0$ on $W$. Since $d\varphi \otimes d\varphi$ has rank at most one, this is 
only possible if $d \varphi=0$ on $W$. Since $\varphi$ is harmonic, this implies $\varphi$ is constant on $M$. Then $L_g^*u=0$, so $g$ is static. As stated above, static metrics have constant scalar curvature on a connected component, a contradiction.
\end{proof}

We now proceed to give a local interpretation of harmonic-static metrics that essentially generalizes their interpretation in Theorem \ref{thm_critical} as critical points of capacity subject to the zero scalar curvature constraint with fixed boundary conditions. To this end, let $\Omega$ be a compact 3-manifold with $\partial \Omega \neq \emptyset$. Let $f \in C^\infty(\partial M)$ be a fixed function. By basic results in elliptic theory, given a Riemannian metric $g$ on $\Omega$, there exists a unique $g$-harmonic function $\varphi_f$ on $\Omega$ that restricts to $f$ on $\partial \Omega$. Let $\gamma$ and $H$ be a fixed Riemannian metric and a fixed smooth function on $\partial \Omega$, respectively. Let $\M_{(\gamma,H)}(\Omega)$ be the set of Riemannian metrics $g$ on $\Omega$ such that $g|_{T\partial \Omega} = \gamma$, and $H_{\partial \Omega} = H$, where $H_{\partial \Omega}$ is the mean curvature of $\partial \Omega$ with respect to $g$, in the direction pointing into $\Omega$.

With boundary data $(\gamma, H)$ and $f$ fixed, we let $\F_{(\gamma, H, f)}: \M_{(\gamma,H)}(\Omega) \to \R$ be defined by
$$\F_{(\gamma, H, f)}(g) = \int_\Omega |\nabla \varphi_f|_g^2 dV_g,$$
where again $\varphi_f$ is $g$-harmonic with Dirichlet boundary data $f$. Note that the right-hand side could also be written as the infimum of $\int_\Omega |\nabla \phi |_g^2 dV_g$ among test functions $\phi$ with $\phi|_{\partial \Omega} = f$. Thus, $\F$ generalizes the capacity in the AF case (in which $f$ is zero on the compact ``inner'' boundary, and there is an ``outer'' boundary at infinity in the AF end where $f \to 1$).

By analogy with the maximum capacity, it is natural to ask if $\F_{(\gamma, H, f)}$ can be maximized on the subset of $\M_{(\gamma,H)}(\Omega)$ consisting of metrics with nonnegative scalar curvature (if this subset is nonempty). It is not clear if this supremum is finite. However, any maximizer would have to have zero scalar curvature by a similar argument as in Proposition \ref{prop_ZSC}. This motivates considering $\F_{(\gamma, H, f)}$ restricted to the subset of $\M_{(\gamma,H)}(\Omega)$ consisting of scalar-flat metrics. We can at last provide a geometric/analytic interpretation of a harmonic-static metric. We omit the analytic details, but the proof is essentially the same as that of Theorem \ref{thm_critical}.

\begin{prop}
Suppose $(\Omega, g)$ is a compact Riemannian 3-manifold with nonempty boundary and zero scalar curvature, and let $\varphi$ be a $g$-harmonic function on $\Omega$. Let $f = \varphi|_{\partial \Omega}$, and let $(\gamma, H)$ be the Bartnik boundary data. Then $(\Omega, g)$ is harmonic-static with respect to $\varphi$ if and only if $g$ is a critical point of $\F_{(\gamma, H, f)}$ on the space of  zero-scalar curvature Riemannian metrics in $\M_{(\gamma,H)}(\Omega)$.
\end{prop}

\section{Examples}
\label{sec_examples}

We recall from Proposition \ref{prop_round} that round spheres with nonnegative constant mean curvature have maximal capacity extensions given by rotationally symmetric exterior regions in a Schwarzschild space (which includes Euclidean space in the special case $m=0$). The first two examples below determine the potential ``$u$'' guaranteed by Theorem \ref{thm_main} for these spaces. Although Example 1 (Euclidean space) is a particular case of Example 2 (Schwarzschild space), we include it separately because the computations are simpler.

\medskip
\paragraph{\emph{Example 1:}}
Let $(M,g)$ be the complement of an open ball in Euclidean space of radius $r_0>0$. We have capacitary potential
$$\varphi(x) = 1-\frac{r_0}{r},$$
where $r=|x|$. 
Direct calculation gives
\begin{align*}
-d\varphi \otimes d\varphi +\frac{1}{2} |\nabla \varphi|^2 g 
&= -\frac{1}{2}r_0^2r^{-4} dr \otimes dr +\frac{1}{2} r_0^2r^{-2} d\sigma^2,
\end{align*}
where $d\sigma^2$ is the standard metric on the unit 2-sphere. Assuming $u=u(r)$, we have 
\begin{align*}
\Hess(u) &= u''(r) dr^2 + ru'(r) d\sigma^2\\
\Delta u &= u''(r) + \frac{2}{r} u'(r).
\end{align*}
Then:
\begin{align*}
L_g^*u &=\Hess(u) - (\Delta u)g -u\Ric 
= - \frac{2}{r} u'(r) dr^2  - \left(ru'(r)+u''(r)r^2\right) d\sigma^2.
\end{align*}
It is easy to verify that $u=-\frac{r_0^2}{8r^2}$ solves $L_g^*u=-d\varphi \otimes d\varphi +\frac{1}{2} |\nabla \varphi|^2 g$. The kernel of $L_g^*$ is four-dimensional, spanned by 1 and the coordinate functions. Therefore the general solution is $u$ plus any affine function. In particular, $u$ is the unique solution that decays to zero at infinity.
\medskip

\paragraph{\emph{Example 2:}}
In this example, let $m \in \R$ and choose $r_0 \geq \frac{m}{2}$ if $m>0$ or $r_0 > \frac{|m|}{2}$ if $m \leq 0$. Let $M = \R^3 \setminus B(0, r_0)$ be equipped with the Schwarzschild metric of mass $m$:
$$g_{ij}= \left(1+ \frac{m}{2r}\right)^4 \delta_{ij},$$
where $r=|x|$.
The capacitary potential is given by
$$\varphi = \frac{1-\frac{r_0}{r}}{1+\frac{m}{2r}}.$$
Direct computation shows:
$$-d\varphi \otimes d\varphi +\frac{1}{2} |\nabla \varphi|^2 g = \frac{(r_0+\frac{m}{2})^2}{2r^4\left(1+\frac{m}{2r}\right)^4 }\left( - dr^2 + r^2 d\sigma^2 \right).$$

To obtain the expression for $L_g^*$, we use the fact that $g$ is conformally flat. For $g_{ij}=w^4 \delta_{ij}$ with $w=w(r)$, we have for a function $u=u(r)$,
\begin{align*}
\Hess_g(u) &= (u'' - 2w^{-1}w'u')dr^2 + (r^{-1}u'+2w^{-1}w'u')r^2d\sigma^2\\
\Delta_g(u) &= w^{-4} \left(u'' + \frac{2}{r}u' + 2w^{-1} w'u' \right)\\
\Ric_g &=-(4w^{-1}w''+4r^{-1}w^{-1}w' -4w^{-2} (w')^2 )dr^2 \\
&\qquad- (2w^{-1}w'' + 6r^{-1}w^{-1} w' + 2w^{-2}(w')^2)r^2d\sigma^2.
\end{align*}
Then
\begin{align*}
L_g^*u &= \Big(-4w^{-1}w'u' - 2r^{-1}u' + u\left(4w^{-1}w'' + 4r^{-1}w^{-1}w'-4w^{-2}(w')^2\right) \Big)dr^2\\ 
&\qquad +\Big(-u''-r^{-1}u' +u\left(2w^{-1}w'' + 6r^{-1}w^{-1}w'+2w^{-2}(w')^2\right) \Big)r^2d\sigma^2.
\end{align*}
Using $w= 1+\frac{m}{2r}$, the harmonic-static equation $L_g^*u = -d\varphi \otimes d\varphi +\frac{1}{2} |\nabla \varphi|^2 g$ reduces to an ODE system for a radial solution $u(r)$ whose general solution (with assistance from Mathematica) is:
$$u(r) = -\frac{(m+2r_0)^2}{32r^2\left(1+\frac{m}{2r}\right)^2} + C \left(\frac{1-\frac{m}{2r}}{1+\frac{m}{2r}}\right)$$
for a constant $C$. Again, the first part is the unique solution that decays to zero at infinity, and the second part comes from the kernel of $L_g^*$ (which includes the static potential $\frac{1-\frac{m}{2r}}{1+\frac{m}{2r}}$).  

\medskip
\paragraph{\emph{Example 3:}} On Euclidean $\R^3$, there are choices of harmonic function other than constants and $\varphi = 1-\frac{c}{r}$ (as in Example 1) that give rise to a solution $u$ to the harmonic-static equation. For example, if $\varphi = z$, then one can check that $u = \frac{1}{8}\left(x^2+y^2 - 3z^2 \right)$ is a solution to $L_g^*u=-d\varphi \otimes d\varphi +\frac{1}{2} |\nabla \varphi|^2 g$ when $g$ is the Euclidean metric. 
(However, there is no reason to expect the Euclidean metric to be harmonic-static with respect to most harmonic functions.)

\medskip
\paragraph{\emph{Example 4:}} Consider the unit 3-sphere with opposite poles removed, i.e.
$$g = dr^2 + \cos(r)^2 d\sigma^2$$
for $r \in \left(-\frac{\pi}{2}, \frac{\pi}{2}\right)$.
With similar calculations as in Example 1, one can check that $\varphi(r) = \tan(r)$ is $g$-harmonic and that $g$ is harmonic-static with respect to $\varphi$. Explicit calculation of $L_g^*u$ with $u=u(r)$ and of $-d\varphi \otimes d\varphi +\frac{1}{2} |\nabla \varphi|^2 g$ yields the following ODEs for the $dr^2$ and $d\sigma^2$ components, respectively:
\begin{align*}
\tan(r)u' - u &= -\frac{1}{4} \sec^4(r)\\
-u''+\tan(r)u'-2u &= \frac{1}{2} \sec^4(r).
\end{align*}
A common solution can be obtained explicitly, with help from Mathematica, as:
$$u(r) =C\sin(r)+  \frac{1}{8}\left(3- \sec^2(r)\right)+\frac{3}{8}\sin(r)\left( \log (\cos(r)) - \log(1+\sin(r))\right).$$
which goes to $-\infty$ at the poles $r=\pm \frac{\pi}{2}$. The $C\sin(r)$ term corresponds to the part of the 4D kernel of $L_g^*$ that can be written in terms of $r$.

\section{Discussion and questions}
\label{sec_discussion}

\subsection{Quasi-local capacity}
Speculatively, we note that whereas Bartnik's definition $m_B(\Sigma, \gamma,H)$ is interpreted as a quasi-local \emph{mass} (along with the Hawking mass \cite{Haw}, Brown--York mass \cite{BY}, Wang--Yau mass \cite{WY}, and many others), we suggest that $\C(\Sigma, \gamma,H)$ could be viewed as a ``quasi-local capacity,'' which as far as we know is a new concept\footnote{Bartnik originally drew a comparison between his mass and the classical electrostatic capacity on $\R^3$, describing the former as a ``nonlinear geometric capacity'' \cite{Ba1}. Occasionally in the literature Bartnik's mass will be be referred to as a capacity, e.g. \cite{HI}. We avoid this here to reduce confusion.}. In the way that quasi-local mass attempts to define an inherently global quantity (total mass) on a finite region, quasi-local capacity could serve the same role for capacity.

Are there other quasi-local capacities? The right-hand side of Bray and Miao's estimate \eqref{eqn_upper_bound}, i.e.,
\begin{equation}
\label{qlc_bm}
\sqrt{\frac{|\Sigma|_\gamma}{16\pi}}\left(1+\sqrt{\frac{ 1}{16\pi}\int_{\Sigma} H^2 dA } \right)
\end{equation}
 seems to be a candidate, as it depends only the Bartnik data and produces the ``correct'' capacity for rotationally symmetric regions bounding the horizon in Schwarzschild space. Note the analogy: the Bartnik mass $m_B$ is to the maximum capacity $\C$ as the 
Hawking mass
$$\sqrt{\frac{|\Sigma|_\gamma}{16\pi}}\left(1- \frac{ 1}{16\pi}\int_{\Sigma} H^2 dA  \right)$$
 is to \eqref{qlc_bm}.
 
In general, just as Bartnik formulated a list of desirable properties for a quasi-local mass \cite{Ba1}, one can readily list analogous properties for quasi-local capacity ``$c$.'' These would include positivity $c(\Omega) > 0$ and monotonicity ($\Omega_1 \subseteq \Omega_2$ implies $c(\Omega_1) \leq c(\Omega_2)$) but exclude rigidity (as $c(\Omega)$ should never vanish). One difference between mass and capacity is the behavior at infinity: quasi-local masses should limit to a finite value (the total mass), but quasi-local capacity should grow without bound. More precisely, if $B_r$ is a coordinate ball of large radius $r$, then $\lim_{r \to \infty} \frac{c(B_r)}{r}$ should equal $1$. This property holds for \eqref{qlc_bm}, and we suspect it should hold for $\C$.

\subsection{Upper semicontinuity of capacity}
As described in the introduction, the Bartnik problem seeks to minimize the ADM mass. This inspired investigations into the lower semicontinuity of ADM mass \cites{J_LSC, J_LSC2, JL1,JL2}. In the context of this paper, we are maximizing capacity, so it is natural to wonder about whether capacity as a functional on the underlying metric satisfies \emph{upper} semicontinuity. Theorems along these lines were recently established in \cite{JPP} for a general class of metric spaces (local integral current spaces) under a relatively weak notion of convergence: pointed Sormani--Wenger intrinsic flat convergence \cite{SW}.

\subsection{Questions and conjectures}
We have left a number of questions unanswered and describe some of them here. First, if $(\Sigma, \gamma, H)$ is the Bartnik data for a round sphere in Euclidean space, then we know that the complement of a Euclidean ball is a maximal capacity extension (Proposition \ref{prop_round}). Additionally, 
for a bounded (and, for simplicity, say) convex domain in Euclidean space, its complement is a minimal mass extension (in the sense of the Barntik mass), which follows from the positive mass theorem with corners \cite{Miao_corners,ST}. Based on these observations, we ask:
if $(\Sigma, \gamma, H)$ is the Bartnik data arising from the boundary of a smooth open convex region $\Omega$ in $\R^3$, is $\R^3 \setminus \Omega$ (with the Euclidean metric) is a maximal capacity extension of $(\Sigma, \gamma, H)$? In light of Theorem \ref{thm_main}, it would be necessary for $\R^3 \setminus \Omega$ to be harmonic-static. Even this is not clear and may be false: as indicated in Example 3 in section \ref{sec_examples}, there is no reason to expect Euclidean space to be harmonic-static with respect to arbitrary harmonic functions. Related to this, it would be interesting to find examples of harmonic-static metrics that are not static.

Next, from the point of view of maximizing capacity, it would intuitively not be optimal for an optimizer of $\C$ to possess a horizon that enclosed $\partial M$, or any surface of less area enclosing $\partial M$ --- it would be more efficient for the admissible extension to ``open up'' more rapidly. Hence:
\begin{conj}
Let $(M,g)$ be a maximal capacity extension of $(\Sigma, \gamma, H)$. Then $M$ has no minimal surfaces in its interior that enclose $\partial M$, and $\partial M$ is outward-minimizing.
\end{conj}

The exact value of the Bartnik mass is known is very few cases, but upper bounds have been found by constructing admissible extensions with controlled ADM mass (for example, in \cites{CCMM, MWX, CM, MP}). It would be interesting to find corresponding lower bounds on $\C$ by constructing admissible extensions and estimating their capacity.  It may also  be possible to compute $\C$ for horizons using something along the lines of the Mantoulidis--Schoen \cite{MS} construction refined in \cite{CCMM}. A natural question is:
\begin{question}
\label{q_horizon}
Suppose  $\Sigma$ is a topological 2-sphere with a Riemannian metric $\gamma$ such that $\lambda_1(-\Delta_\gamma + K)>0$. Is $\C(\Sigma, \gamma, 0)$ (the horizon case) equal to $\sqrt{\frac{|\Sigma|_{\gamma}}{16\pi}}$?
\end{question}
The inequality $\C(\Sigma, \gamma, 0) \leq \sqrt{\frac{|\Sigma|_{\gamma}}{16\pi}}$ follows from Proposition \ref{prop_C_upper_bound}. Proving the opposite inequality would entail producing admissible extensions of $(\Sigma, \gamma,H)$ whose capacity is arbitrarily close to $\sqrt{\frac{|\Sigma|_{\gamma}}{16\pi}}$, and it is not clear this is possible. For example, it is conceivable that the deficit $\sqrt{\frac{|\Sigma|_{\gamma}}{16\pi}} - \C(\Sigma, \gamma, 0)$ is a function of how far $\gamma$ is from roundness. In section \ref{app_minimizing} we recall some details of Mantouldis--Schoen construction, though we use it in a way that is far from optimal from the point of view of Question \ref{q_horizon}.

\section{Appendix}

\subsection{Necessity of hypotheses}
In this first part of the appendix, we point out that both nonnegativity of scalar curvature and the imposition of the mean curvature boundary condition are necessary to obtain a meaningful maximum capacity $\C$.

\begin{prop}
\label{prop_hypotheses}
If either the boundary mean curvature condition $H \geq H_{\partial M}$ or the nonnegativity  of scalar curvature assumption were dropped in the definition of admissible extension in Definition \ref{def_extension}, then $\C(\Sigma,\gamma,H)$ would only take on the values of $\pm \infty$.
\end{prop}
\begin{proof}
With $(\Sigma,\gamma,H)$ given, suppose that an AF extension $(M,g)$ with $R_g \geq 0$ exists (with no restriction on the mean curvature). (If this is not the case, then $\C(\Sigma,\gamma,H)$ as modified would simply be $-\infty$.) We claim that  AF extensions with $R_g \geq 0$ and arbitrarily large capacity exist.

To see this, let $\varphi$ be the capacitary potential of $(M,g)$. For a parameter $c>0$, let $u_c = 1 + (c-1)\varphi$. Thus, $\Delta_g u_c = 0$, $u_c=1$ on $\partial M$, and $u_c \to c$ at infinity. By the maximum principle, $u_c>0$. The conformal metric $\ol g_c = u_c^4 g$ is AF, and its boundary is isometric to $(\Sigma, \gamma)$. Since $u_c$ is harmonic and $R_g \geq 0$, it follows, that $\ol g_c$ likewise has nonnegative scalar curvature.

We now compute the capacity of $\partial M$ in $(M, \ol g_c)$. Using equation (239) of \cite{Bray_RPI}, one can see that $\frac{\varphi}{u_c}$  is harmonic with respect to $\ol g_c$. In particular,  $\frac{c\varphi}{u_c}$ is $\ol g_c$-harmonic with boundary conditions of 0 and 1 on $\partial M$ and at infinity and is therefore the capacitary potential of $(M, \ol g_c)$. We have (using the fact that $g=\ol g_c$ on $\partial M$):
\begin{align*}
\capac_{\ol g_c}(\partial M) &= \frac{1}{4\pi} \int_{\partial M} \partial_{\nu} \left(\frac{c\varphi}{u_c} \right) dA\\
&= \frac{c}{4\pi} \int_{\partial M} \partial_{\nu}(\varphi) dA\\
&=c \capac_g(\partial M).
\end{align*}
Since $c$ can be arbitrarily large, so can the capacity of the extension if the mean curvature condition is dropped. (Indeed, the mean curvature of $\partial M$ with respect to $\ol g_c$ also becomes arbitrarily large.)

Next, again with $(\Sigma,\gamma,H)$ given, suppose that an AF extension $(M,g)$ with ``$H \geq H_{\partial M}$'' exists (but where $R_g$ is not necessarily nonnegative).
(If this is not the case, then $\C(\Sigma,\gamma,H)$ as modified would simply be $-\infty$.) We claim that AF extensions that respect ``$H \geq H_{\partial M}$'' and have arbitrarily large capacity exist.

To see this, let $\Gamma>0$ be large, and fix $c$ as  in the previous argument (and the corresponding $u_c$) so that $\capac_{\ol g_c}(\partial M) > \Gamma$. For any $\epsilon>0$ small, let $0 \leq \rho_\epsilon \leq 1$ be 
a smooth function on $M$ that equals $1$ in a neighborhood of $\partial M$ and is supported in an $\epsilon$-neighborhood of $\partial M$. Let 
$v_\epsilon = \rho_\epsilon + (1-\rho_\epsilon) u_c$. Since $u_c=1$ on $\partial M$, it is straightforward to verify that $v_\epsilon$ converges uniformly to $u_c$ as $\epsilon \searrow 0$ 
(and $v_\epsilon$ agrees with $u_c$ outside a compact set. From this is follows that  $\capac_{v_\epsilon ^4 g}(\partial M)$ converges to $\capac_{\ol g_c}(\partial M)$  as 
$\epsilon \searrow 0$ and thus exceeds $\Gamma$ for $\epsilon>0$ sufficiently small. Since $v_\epsilon=g$ on a neighborhood of $\partial M$, We see the mean curvature condition is 
respected for the extension $(M, v_\epsilon)$. This completes the proof.
\end{proof}

\medskip

\subsection{Minimizing capacity is trivial}
\label{app_minimizing}
As promised in the introduction, we demonstrate here that \emph{minimizing} the capacity of admissible extensions of fixed Bartnik data $(\Sigma, \gamma,H)$ is uninteresting: the result is always zero --- even when restricting to admissible extensions containing no interior minimal surfaces or for which the $\partial M$ strictly minimizes area in its homology class. We assume here the conditions of Proposition \ref{prop_ext_exist} (the ``Mantoulidis--Schoen'' conditions), i.e., $\Sigma$ is topologically $S^2$ and the lowest eigenvalue of  $-\Delta_\gamma + K$ is positive. The intuition here is that the presence of a long, nearly-cylindrical region enclosing the boundary will ensure the capacity is small.

Mantoulidis and Schoen construct in \cite{MS} a Riemannian metric on a collar neighborhood of $\partial M$, i.e. on $[0,1] \times S^2$ with $t=0$ corresponding to $\partial M$,
of the form
$$g= A^2 u(t, \cdot)^2 dt^2 + (1+\epsilon t^2)\gamma(t),$$
where $\gamma(t)$ is a smooth family of Riemannian metrics on $S^2$ with $\gamma(0)=\gamma$, and $t \mapsto u(t,x)>0$ is  smooth family of functions on $S^2$. Here $\epsilon>0$ and $A>0$ are constants, where $A$ can always be increased. The metric $g$ has a minimal surface at $t=0$, with the other spheres $\{t\} \times S^2$ having positive mean curvature for $t \in (0,1]$. Also, $g$ has positive scalar curvature. Their ultimate construction extends $g$ and modifies it on $M$ to produce an admissible extension of $(\Sigma, \gamma,H)$ --- but the modification only occurs for $t \in [1/2,1]$. Thus, we may estimate the capacity using a test function that is constant outside of $S^2 \times [0,1/2]$. We let $\psi$ be the Lipschitz function on $(M,g)$ that equals $2t$ for $t \in [0,1/2]$ and equals 1 elsewhere. Then:
\begin{align*}
\int_M |\nabla \psi|^2_g  dV_g &= \int_{0}^{1/2} \int_{S^2} \frac{4}{A^2 u(t,\cdot)^2} \cdot A u(t,\cdot)(1+\epsilon t^2) dA_{\gamma(t)} dt\\
&= \frac{4}{A} \int_{0}^{1/2} (1+\epsilon t^2) \int_{S^2} \frac{1}{ u(t,\cdot)} dA_{\gamma(t)} dt,
\end{align*}
which converges to $0$ as $A \to \infty$. This shows that capacity of such $(M,g)$ can be made arbitrarily small, yet there are no closed minimal surfaces in the interior, and every surface enclosing the boundary has strictly greater area.

\subsection{Equivalence of definitions}
Here we show the two versions of the maximum capacity employed in this paper (i.e., Definition \ref{def_C} and equation \eqref{eqn_C_tilde}) are equivalent.

\begin{lemma}
\label{lemma_regularity}
Let $\Sigma$ be a smooth, closed, connected, orientable surface embedded in $\R^3$, which we identify with the surface $\partial \Omega$ for a bounded domain $\Omega \subset \R^3$. Let
$$\C(\Sigma, \gamma, H) = \sup_{(M,g)} \left\{ \capac_g(\partial M) \; : \; (M,g) \text{ is an admissible AF extension of } (\Sigma, \gamma, H) \right\}$$
be as in the introduction, and let
$$\tilde \C(\Sigma, \gamma, H) = \sup_{g} \left\{ \capac_g(\partial \Omega) \; : \; g \in \M^{k,\alpha}_{-\tau}(\R^3 \setminus \Omega), R_g \geq 0, g|_\Sigma = \gamma, \text{ and } H \geq H_{\partial M} \right\}$$
be as in \eqref{eqn_C_tilde}. Then $\C(\Sigma, \gamma, H) = \tilde \C(\Sigma, \gamma, H)$ (and is, in particular, independent of $\tau \in (0,\frac{1}{2})$, $k\geq 2$, and $\alpha \in (0,1)$).
\end{lemma}
\begin{proof}
On the one hand, let $(M,g)$ be a test Riemannian manifold for $\C(\Sigma, \gamma, H)$, with decay rate $\tau_0 \in (1/2,1]$ as in Definition \ref{def_AF}. We first claim that the mean curvature boundary condition $H \geq H_{\partial M}$  can be taken to be strict, without loss of generality. For $\delta \in (0,1)$, let $u_\delta>0$ be the harmonic function on $(M,g)$ that equals 1 on $\partial M$ and approaches $1-\delta$ at infinity. (Note $u_\delta$ can be written in terms of the capacitary potential as $u_\delta = 1 - \delta \varphi$.) Then the conformal deformation $(M,u_\delta^4 g)$ is also a test Riemannian manifold for $\C(\Sigma, \gamma,H)$: it has nonnegative scalar curvature because $u_\delta$ is $g$-harmonic, it is AF of order $\tau_0$, it induces boundary metric $\gamma$ on $\partial M$ --- and satisfies ``$H > H_{\partial M}$'' strictly. The latter can be seen since $u|_{\partial M} = 1$ and $\partial_{\nu} u_\delta < 0$ along $\partial M$ by the maximum principle (with unit normal $\nu$ pointing into $M$). Also from the maximum principle, $|u_\delta -1| < \delta$. From this it follows that $\capac_{u_\delta^4 g} (\partial M) \to \capac_g(\partial M)$ as $\delta  \to 0$. This establishes the claim. So we assume the mean curvature boundary condition holds strictly for $g$: $H > H_{\partial M}$.

Next, we claim that $g$ can be taken to have particularly nice asymptotics. Specifically, we appeal to a lemma of Schoen--Yau \cite{SY2}, stated in a convenient form in \cite[Lemma 1]{Bray_RPI}: $g$ can be perturbed to an AF metric $g_\epsilon$ of nonnegative scalar curvature that is \emph{harmonically flat at infinity} (explained below), where the perturbation changes the metric by less than $\epsilon$ in a (global) uniform sense. Thus, the capacity is continuous as $\epsilon \to 0$. Moreover, near the boundary, the construction is carried out by a conformal factor $w$ that equals 1 on $\partial M$ (and thus preserves the boundary metric), with $|\nabla w|$ converging to 0 on $\partial M$ as $\epsilon \to 0$. In particular, for $\epsilon>0$ sufficiently small, $H > H_{\partial M}$ holds with respect to $g_\epsilon$. Thus, $(M,g_\epsilon)$ is a test Riemannian manifold for $\C(\Sigma, \gamma, H)$ --- so we can without loss of generality assume $g$ is harmonically flat at infinity. This means that  $g$ equals $U^4 \delta_{ij}$ outside of a compact set, where $U$ is harmonic with respect to $\delta_{ij}$ and $U \to 1$ at infinity. Using a standard expansion into spherical harmonics, we see that $g$ is AF order 1, and moreover $|\partial_\lambda g_{ij}|$ is $O(|x|^{-1-|\lambda|})$ for any multi-index $\lambda$. In light of remark \ref{rmk_diffeo}, we can without loss of generality take $g$ to be defined on $\R^3 \setminus \Omega$ with the stated decay, i.e., $g \in \M^{k,\alpha}_{-\tau}(\R^3 \setminus \Omega)$ for any $\tau \in (\frac{1}{2},1)$, $k\geq 2$, and $\alpha \in (0,1)$. In particular, $g$ is a test metric for $\tilde \C(\Sigma, \gamma, H)$ --- and the capacity has changed arbitrarily little in the course of this construction.

On the other hand, let $g \in \M^{k,\alpha}_{-\tau}(M)$  (where $M=\R^3 \setminus \Omega$) be a test metric for $\tilde \C(\Sigma, \gamma, H)$ i.e., let $g$ have nonnegative scalar curvature, with the induced metric on $\partial M$  from $g$ isometric to $(\Sigma, \gamma)$, and with $H \geq H_{\partial M}$. We  apply a  similar argument initially: the same steps allow us to assume $H > H_{\partial M}$ without loss of generality. Using a conformal argument (specifically, replacing $g$ with $u^4g$, where $u = 1$ on $\partial M$, $u \to 1$ at infinity, and $\Delta u = -f$, where $f>0$ on $M$ is pointwise small with rapid decay at infinity), we can assume $g$ has strictly positive scalar curvature everywhere without loss of generality. Finally, since smooth functions are dense in the space of $C^2$ functions with the $C^2$ norm, we can replace $g$ with a smooth metric $g'$ (still equaling $\gamma$ on the boundary) with the following properties: 1) $g'$ has positive scalar curvature, 2) $g'$ is AF as in Definition \ref{def_AF} (i.e., with decay up to 2nd derivatives) of order $\tau$, 3) $H > H_{\partial M}$ with respect to $g'$ (since mean curvature only depends on first derivatives of the metric), and 4) $g'$ is uniformly $C^0$-close to $g$. Property 4) here ensures that the capacity of $\partial M$ with respect to $g'$ can be made arbitrarily close to $\capac_g(\partial M)$. In particular, $(M,g')$ is a test Riemannian manifold for $\C(\Sigma, \gamma, H)$.
\end{proof}

\subsection{Variation of capacity under a flow of surfaces}
Finally, we observe that the formula for the derivative of the capacity with respect to the metric in Proposition \ref{prop_Dcap} can be used to give a proof of the variation of the capacity for a flow of surfaces in an AF manifold. This is certainly already known for surfaces in $\R^n$ (see \cite{DMM}, for example).
\begin{prop}
 \label{prop_cap_varying_boundary}
Let $(M,g)$ be an asymptotically flat 3-manifold without boundary, and let $\{\Sigma_t\}_{t \in [0,\epsilon)}$ be a smooth family of closed surfaces in $M$ such that $\Sigma_t = \partial \Omega_t$ for open bounded regions $\Omega_t$ with $M_t :=  M \setminus \Omega_t$ connected. Then:
\begin{equation}
\label{eqn_surfaces}
\frac{d}{dt} \left.\capac_g(\Sigma_t)\right|_{t=0} = \frac{1}{4\pi} \int_{\Sigma_0} |\nabla \varphi|^2 \langle X, \nu \rangle dA,
\end{equation}
where $\varphi$ is the capacitary potential for $M_0$, $X$ is the smooth vector field on $\Sigma_0$ (taking values in $TM$) defined by the flow of $\Sigma_t$ at $t=0$, and $\nu$ is the unit normal to $\Sigma_0$ pointing out of $\Omega_0$ (with all metric quantities taken with respect to $g$).
\end{prop}
\begin{proof}
Extend $X$ to a smooth vector field on $M$ of compact support, and construct a 1-parameter family of diffeomorphisms $\Phi_t: M \to M$ as the flow of $X$ (so $\Phi_0$ is the identity and $\Phi_t$ is the identity outside of a compact set). To compute the desired derivative, it suffices to the compute the derivative of the capacity of $\Sigma_t' := \partial (M \setminus \Phi_t(\Omega_0))$. Note that
$$\capac_g(\Sigma_t') = \capac_{\Phi_t^* g}(\Sigma_0),$$
so the problem is further reduced to considering a fixed space with varying metric. We apply Proposition \ref{prop_Dcap} with $h = \L_X g$, the Lie derivative. The second line below can be verified using local coordinates, using the fact that $d\varphi \otimes d\varphi + \frac{1}{2}|\nabla \varphi|^2 g$ is divergence-free:
\begin{align*}
\frac{d}{dt} \left.\capac_{\Phi_t^* g}(\Omega_0)\right|_{t=0}&=\frac{1}{4\pi} \int_{M_0} \langle -d\varphi \otimes d\varphi + \frac{1}{2}|\nabla \varphi|^2 g, \L_X g\rangle dV\\
&=\frac{1}{2\pi} \int_{M_0} \Div\left[ (-d\varphi \otimes d\varphi + \frac{1}{2}|\nabla \varphi|^2 g)(X, \cdot) \right] dV\\
&=- \frac{1}{2\pi} \int_{\Sigma_0}  (-d\varphi \otimes d\varphi + \frac{1}{2}|\nabla \varphi|^2 g)(X,\nu) \ dV,
\end{align*}
where $\nu$ is the unit normal to $\Sigma$ with respect to $g$, pointing in $M$. Using $\nabla \varphi = |\nabla \varphi| \nu$ along $\Sigma_0$ reduces this to \eqref{eqn_surfaces}.
\end{proof}


\begin{bibdiv}
 \begin{biblist}
 
\bib{An}{article}{
   author={An, Z.},
   title={On mass-minimizing extensions of Bartnik boundary data},
   journal={Comm. Anal. Geom.},
   volume={31},
   date={2023},
   number={6},
   pages={1353--1403}
}
 
 \bib{AJ}{article}{
   author={Anderson, M.},
   author={Jauregui, J.},
   title={Embeddings, immersions and the Bartnik quasi-local mass
   conjectures},
   journal={Ann. Henri Poincar\'{e}},
   volume={20},
   date={2019},
   number={5},
   pages={1651--1698}
}

 \bib{adm}{article}{
   author={Arnowitt, R.},
   author={Deser, S.},
   author={Misner, C.},
   title={Coordinate invariance and energy expressions in general relativity},
   journal={Phys. Rev. (2)},
   volume={122},
   date={1961},
   pages={997--1006},
}

\bib{BE}{article}{
   author={Baird, P.},
   author={Eells, J.},
   title={A conservation law for harmonic maps},
   conference={
      title={Geometry Symposium, Utrecht 1980},
      address={Utrecht},
      date={1980},
   },
   book={
      series={Lecture Notes in Math.},
      volume={894},
      publisher={Springer, Berlin-New York},
   },
   date={1981},
   pages={1--25}
}

\bib{Ba0}{article}{
   author={Bartnik, R.},
   title={The mass of an asymptotically flat manifold},
   journal={Comm. Pure Appl. Math.},
   volume={39},
   date={1986},
   number={5},
   pages={661--693},
}

 
 \bib{Ba1}{article}{
   author={Bartnik, R.},
   title={New definition of quasilocal mass},
   journal={Phys. Rev. Lett.},
   volume={62},
   date={1989},
   number={20},
   pages={2346--2348}
}

\bib{Ba5}{article}{
   author={Bartnik, R.},
   title={Some open problems in mathematical relativity},
   conference={
      title={Conference on Mathematical Relativity},
      address={Canberra},
      date={1988},
   },
   book={
      series={Proc. Centre Math. Anal. Austral. Nat. Univ.},
      volume={19},
      publisher={Austral. Nat. Univ., Canberra},
   },
   date={1989},
   pages={244--268},
   review={\MR{1020805}},
}

\bib{Ba3}{article}{
   author={Bartnik, R.},
   title={Energy in general relativity},
   conference={
      title={Tsing Hua lectures on geometry \&\ analysis},
      address={Hsinchu},
      date={1990--1991},
   },
   book={
      publisher={Int. Press, Cambridge, MA},
   },
   date={1997},
   pages={5--27}
}

\bib{Ba4}{article}{
   author={Bartnik, R.},
   title={Mass and 3-metrics of non-negative scalar curvature},
   conference={
      title={Proceedings of the International Congress of Mathematicians,
      Vol. II},
   },
   book={
      publisher={Higher Ed. Press, Beijing},
   },
   date={2002},
   pages={231--240}
}

\bib{Ba_phase}{article}{
   author={Bartnik, R.},
   title={Phase space for the Einstein equations},
   journal={Comm. Anal. Geom.},
   volume={13},
   date={2005},
   number={5},
   pages={845--885},
   issn={1019-8385},
   review={\MR{2216143}},
}



 
 \bib{Bray_RPI}{article}{
   author={Bray, H.},
   title={Proof of the Riemannian Penrose inequality using the positive mass
   theorem},
   journal={J. Differential Geom.},
   volume={59},
   date={2001},
   number={2},
   pages={177--267}
}

 \bib{BM}{article}{
   author={Bray, H.},
   author={Miao, P.},
   title={On the capacity of surfaces in manifolds with nonnegative scalar
   curvature},
   journal={Invent. Math.},
   volume={172},
   date={2008},
   number={3},
   pages={459--475}
}

\bib{BY}{article}{
   author={Brown, J.},
   author={York, J.},
   title={Quasilocal energy and conserved charges derived from the
   gravitational action},
   journal={Phys. Rev. D (3)},
   volume={47},
   date={1993},
   number={4},
   pages={1407--1419}
}

\bib{CCMM}{article}{
   author={Cabrera Pacheco, A.},
   author={Cederbaum, C.},
   author={McCormick, S.},
   author={Miao, P.},
   title={Asymptotically flat extensions of CMC Bartnik data},
   journal={Classical Quantum Gravity},
   volume={34},
   date={2017},
   number={10},
   pages={105001, 15}
}
 
 
  \bib{CM}{article}{
   author={Chau, A.},
   author={Martens, A.},
   title={ On the Bartnik mass of non-negatively curved CMC spheres},
   eprint={https://arxiv.org/abs/2102.03632},
}

 
\bib{CLN}{book}{
   author={Chow, B.},
   author={Lu, P.},
   author={Ni, L.},
   title={Hamilton's Ricci flow},
   series={Graduate Studies in Mathematics},
   volume={77},
   publisher={American Mathematical Society, Providence, RI; Science Press
   Beijing, New York},
   date={2006}
}

\bib{Ch}{article}{
   author={Chru\'sciel, P.},
   title={Boundary conditions at spatial infinity from a Hamiltonian point
   of view},
   conference={
      title={Topological properties and global structure of space-time},
      address={Erice},
      date={1985},
   },
   book={
      series={NATO Adv. Sci. Inst. Ser. B Phys.},
      volume={138},
      publisher={Plenum, New York},
   },
   date={1986},
   pages={49--59}
}


\bib{Cor}{article}{
   author={Corvino, J.},
   title={Scalar curvature deformation and a gluing construction for the
   Einstein constraint equations},
   journal={Comm. Math. Phys.},
   volume={214},
   date={2000},
   number={1},
   pages={137--189}
}

\bib{Cor2}{article}{
   author={Corvino, J.},
   title={A note on the Bartnik mass},
   conference={
      title={Nonlinear analysis in geometry and applied mathematics},
   },
   book={
      series={Harv. Univ. Cent. Math. Sci. Appl. Ser. Math.},
      volume={1},
      publisher={Int. Press, Somerville, MA},
   },
   date={2017},
   pages={49--75}
}

\bib{DMM}{article}{
   author={De Philippis, G.},
   author={Marini, M.},
   author={Mukoseeva, E.},
   title={The sharp quantitative isocapacitary inequality},
   journal={Rev. Mat. Iberoam.},
   volume={37},
   date={2021},
   number={6},
   pages={2191--2228}
}

\bib{DM}{book}{
   author={Dr\'{a}bek, P.},
   author={Milota, J.},
   title={Methods of nonlinear analysis},
   series={Birkh\"{a}user Advanced Texts: Basler Lehrb\"{u}cher.
   [Birkh\"{a}user Advanced Texts: Basel Textbooks]},
   edition={2},
   note={Applications to differential equations},
   publisher={Birkh\"{a}user/Springer Basel AG, Basel},
   date={2013},
   pages={x+649}
}

\bib{FM}{article}{
   author={Fischer, A.},
   author={Marsden, J.},
   title={Deformations of the scalar curvature},
   journal={Duke Math. J.},
   volume={42},
   date={1975},
   number={3},
   pages={519--547}
}

\bib{Haw}{article}{
   author={Hawking, S.},
   title={Gravitational radiation in an expanding universe},
   journal={J. Mathematical Phys.},
   volume={9},
   date={1968},
   number={4},
   pages={598--604}
}

\bib{HI}{article}{
   author={Huisken, G.},
   author={Ilmanen, T.},
   title={The inverse mean curvature flow and the Riemannian Penrose
   inequality},
   journal={J. Differential Geom.},
   volume={59},
   date={2001},
   number={3},
   pages={353--437},
}

\bib{J_LSC}{article}{
   author={Jauregui, J.},
   title={On the lower semicontinuity of the ADM mass},
   journal={Comm. Anal. Geom.},
   volume={26},
   date={2018},
   number={1},
   pages={85--111}
}

\bib{J_LSC2}{article}{
   author={Jauregui, Jeffrey L.},
   title={Lower semicontinuity of the ADM mass in dimensions two through
   seven},
   journal={Pacific J. Math.},
   volume={301},
   date={2019},
   number={2},
   pages={441--466}
}

\bib{Jau}{article}{
   author={Jauregui, J.},
   title={Smoothing the Bartnik boundary conditions and other results on
   Bartnik's quasi-local mass},
   journal={J. Geom. Phys.},
   volume={136},
   date={2019},
   pages={228--243}
}

\bib{JL1}{article}{
   author={Jauregui, J.},
   author={Lee, D.},
   title={Lower semicontinuity of mass under $C^0$ convergence and Huisken's
   isoperimetric mass},
   journal={J. Reine Angew. Math.},
   volume={756},
   date={2019},
   pages={227--257}
}

\bib{JL2}{article}{
   author={Jauregui, J.},
   author={Lee, D.},
   title={Lower semicontinuity of ADM mass under intrinsic flat convergence},
   journal={Calc. Var. Partial Differential Equations},
   volume={60},
   date={2021},
   number={5},
   pages={Paper No. 193, 42}
}

\bib{JPP}{article}{
   author={Jauregui, J.},
   author={Perales, R.},
   author={Portegies, J.},
   title={Semicontinuity of capacity under pointed intrinsic flat
   convergence},
   journal={Comm. Anal. Geom.},
   volume={33},
   date={2025},
   number={3},
   pages={559--621}
}

\bib{MS}{article}{
   author={Mantoulidis, C.},
   author={Schoen, R.},
   title={On the Bartnik mass of apparent horizons},
   journal={Classical Quantum Gravity},
   volume={32},
   date={2015},
   number={20},
   pages={205002, 16}
}

\bib{MMT}{article}{
   author={Mantoulidis, C.},
   author={Miao, P.},
   author={Tam, L.-F.},
   title={Capacity, quasi-local mass, and singular fill-ins},  
   journal={J. Reine Angew. Math. (to appear)}
}

\bib{McC2}{article}{
   author={McCormick, S.},
   title={An overview of Bartnik’s quasi-local mass},
   journal={Beijing J. of Pure and Appl. Math.},
   volume={1},
   date={2024},
   number={2},
   pages={455--487}
}


\bib{McC}{article}{
   author={McCormick, S.},
   title={Gluing Bartnik extensions, continuity of the Bartnik mass, and the
   equivalence of definitions},
   journal={Pacific J. Math.},
   volume={304},
   date={2020},
   number={2},
   pages={629--653}
}


\bib{Miao_corners}{article}{
   author={Miao, P.},
   title={Positive mass theorem on manifolds admitting corners along a
   hypersurface},
   journal={Adv. Theor. Math. Phys.},
   volume={6},
   date={2002},
   number={6},
   pages={1163--1182}
}

\bib{Miao_boundary}{article}{
   author={Miao, P.},
   title={Variational effect of boundary mean curvature on ADM mass in
   general relativity},
   conference={
      title={Mathematical physics research on the leading edge},
   },
   book={
      publisher={Nova Sci. Publ., Hauppauge, NY},
   },
   date={2004},
   pages={145--171}
}

\bib{MP}{article}{
   author={Miao, P.},
   author={Piubello, A.},
   title={Estimates of the Bartnik mass},
   journal={Beijing J.Pure Appl.Math.},
   volume={1},
   date={2024},
   number={2},
   pages={489-513}
}


\bib{MWX}{article}{
   author={Miao, P.},
   author={Wang, Y.},
   author={Xie, N.},
   title={On Hawking mass and Bartnik mass of CMC surfaces},
   journal={Math. Res. Lett.},
   volume={27},
   date={2020},
   number={3},
   pages={855--885}
}

\bib{SY2}{article}{
   author={Schoen, R.},
   author={Yau, S.-T.},
   title={Proof of the positive mass theorem. II},
   journal={Comm. Math. Phys.},
   volume={79},
   date={1981},
   number={2},
   pages={231--260}
}


\bib{ST}{article}{
   author={Shi, Y.},
   author={Tam, L.-F.},
   title={Positive mass theorem and the boundary behaviors of compact
   manifolds with nonnegative scalar curvature},
   journal={J. Differential Geom.},
   volume={62},
   date={2002},
   number={1},
   pages={79--125}
}

\bib{WY}{article}{
   author={Wang, M.-T.},
   author={Yau, S.-T.},
   title={Isometric embeddings into the Minkowski space and new quasi-local
   mass},
   journal={Comm. Math. Phys.},
   volume={288},
   date={2009},
   number={3},
   pages={919--942}
}

\bib{SW}{article}{
   author={Sormani, C.},
   author={Wenger, S.},
   title={The intrinsic flat distance between Riemannian manifolds and other
   integral current spaces},
   journal={J. Differential Geom.},
   volume={87},
   date={2011},
   number={1},
   pages={117--199}
}



\end{biblist}
\end{bibdiv}
\end{document}